\documentclass[a4paper,12pt]{article}
\usepackage{cmap}                        
\usepackage[cp1251]{inputenc}            
\usepackage[english]{babel}
\usepackage[left=2cm,right=2cm,top=2cm,bottom=2cm]{geometry} 
\usepackage{amssymb}
\usepackage{amsmath, amsthm}
\usepackage{indentfirst}
\theoremstyle{plain}
\newtheorem{thm}{Theorem}[section]
\newtheorem{lem}{Lemma}[section]

\newtheorem{cor}{Corollary}[section]

\theoremstyle{definition}

\newtheorem{definition}{Definition}
\newtheorem{pr}{Question}

\begin{document}

\begin{center}\Large
\textbf{On Autonilpotent Finite Groups}
\normalsize

\smallskip
V.\,I. Murashka

 \{mvimath@yandex.ru\}

 Francisk Skorina Gomel State University, Gomel, Belarus\end{center}

\textbf{Abstract.} In the paper autonilpotent groups were characterized as groups $G$  such that $\mathrm{Aut}G$  stabilizes some chain of subgroups of $G$. It was shown that a $p$-group is autonilpotent if and only if its group of automorphisms is also a $p$-group.  Analogues of Baer's theorem about the hypercenter and Frobenius $p$-nilpotency criterion were obtained for autonilpotent groups.


 \textbf{Keywords.} Finite groups; nilpotent groups; autonilpotent groups;   hypercenter of a group.

\textbf{AMS}(2010). 20D25, 20D45, 20D15.

\section{Introduction and results}

Throughout this paper and all groups are finite and   $G$ always denotes a finite group. Recall that $\mathrm{Aut}G$ and $\mathrm{Inn}G$  are the groups of all  and inner automorphisms of $G$ respectively.

Let $A$ be a group of automorphisms of a group $G$. Kaloujnine \cite{i0} and Hall \cite{i1}   showed that if $A$ stabilizes some chain of subgroups of $G$ then $A$ is nilpotent.  In this paper we will consider the converse question: assume that $A$ stabilizes some chain of subgroups of $G$, what can be said about $G?$

M. R. R. Moghaddam and M. A. Rostamyari \cite{a0} introduced the concept of autonilpotent group. Let   \[L_n(G)=\{ x\in G\,\,| \,\,[x, \alpha_1,\dots, \alpha_n]=1 \,\,\forall \alpha_1,\dots, \alpha_n\in\mathrm{Aut}G\}\]
Then $G$  is called autonilpotent if $G=L_n(G)$ for some natural $n$. Some properties of autonilpotent groups were studied in \cite{a1}.


\begin{thm}\label{th1}
A group $G$ is   autonilpotent if and only if  $\mathrm{Aut}G$ stabilizes some chain of subgroups of $G$.\end{thm}

 Recall that $\pi(G)$ is the set of prime divisors of $G$.  It is known that a group is nilpotent if and only if it is the direct product of its Sylow subgroups. Here we proved

 \begin{thm}\label{th1}
A group $G$ is   autonilpotent if and only if it is  the direct product of its Sylow subgroups and the automorphism group of a Sylow $p$-subgroup of $G$ is a $p$-group for all  $p\in\pi(G)$.\end{thm}


In \cite{a2} all abelian autonilpotent groups were described.  In particular, abelian autonilpotent non-unit groups of odd order don't exist.  It was not known about the existence of autonilpotent $p$-groups of odd order.

\begin{cor} Let $p$ be a prime. A $p$-group $G$ is autonilpotent if and only if $\mathrm{Aut}G$ is a $p$-group. \end{cor}

An example of a $p$-group $G$ of order $p^5$ $(p>3)$ such that $\mathrm{Aut}G$ is also a $p$-group  was constructed in \cite{a3}. In the library of small groups  of GAP \cite{h0} there are 30 groups of order $3^6$ such that their automorphism groups are also $3$-groups (for example groups [729, 31], [729, 41] and [729, 46]).

Note that $L_n(G)\subseteq L_{n+1}(G)$ for all natural $n$. Since $G$ is finite, it is clear that $L_n(G)=L_{n+1}(G)$ for some natural $n$. In this case we shall call a subgroup $L_n(G)=L_\infty(G)$ the \emph{absolute hypercenter} of $G$. Hence $G$  is   autonilpotent if and only if $G=L_\infty(G)$. In \cite{h1} Baer  showed that a $p$-element $g$ of  $G$ belongs to the hypercenter $Z_\infty(G)$ of $G$ if and only if it commutes with all $p'$-elements of $G$. It means that $g^x=g$ for any $p'$-element $x$ of $G$. Here we obtain the following analogue of Baer's result

\begin{thm}\label{t2}Let $g$ be a $p$-element of a group $G$. Then $g\in L_\infty(G)$ if and only if $g^\alpha=g$ for every $p'$-element $\alpha$ of $\mathrm{Aut}G$.      \end{thm}

\begin{cor}
 A group $G$ is   autonilpotent if and only if    every  automorphism $\alpha$  of $G$   fixes all elements of $G$  whose orders are coprime to the order of $\alpha$.
\end{cor}

According to Frobenius $p$-nilpotency criterion (see \cite[5E, 5.26]{is}) a group $G$ is nilpotent if and only if $N_G(P)/C_G(P)$ is a $p$-group for every  $p$-subgroup $P$ of $G$ and every    $p\in\pi(G)$.

\begin{thm}\label{t3}  A group $G$ is autonilpotent if and only if $N_{\mathrm{Aut}G}(P)/C_{\mathrm{Aut}G}(P)$ is a $p$-group for every  $p$-subgroup $P$ of $G$ and every $p\in\pi(G)$.     \end{thm}

\section{Preliminaries}


For an automorphism $\alpha$ and an element $x$ of a group $B$ we denote by $x^\alpha$ the image of $x$ under $\alpha$.       Let $A$ and $B$ be groups and $\varphi$ be a homomorphism form $A$ to $\mathrm{Aut}B$. We can define the action of $A$ on $B$ in the following way
\[ x^a=x^{\varphi(a)}, \,\, x\in B, \,\, a\in A.\]
In this case $A$ is called a \emph{group of operators} of $B$. It is known that the following sets are subgroups of $A$ for any subgroup $D$ of $B$.
$$ N_A(D)=\{a\in A \,|\, D^a=D\}, \,\,   C_A(D)=\{a\in A \,|\, d^a=d \,\,\forall d\in D\}\,\,\mathrm{and} \,\,\mathrm{Aut}_A(D)=N_A(D)/C_A(D),$$
where $\mathrm{Aut}_A(D)$ is a group of automorphisms  induced by $A$ on  $D$.
Note that $\mathrm{Ker}\varphi=C_A(B)$. Hence the actions of $A$  and $\mathrm{Aut}_A(B)=A/C_A(B)$ on $B$ are the same.

Recall that   $[b, a]=b^{-1}b^a$ for $a\in A$ and $b\in B$. Let $y\in B$. We can also consider $y$ as an element of $\mathrm{Inn}B$. Assume that $y\in N_{\mathrm{Aut}B}(\mathrm{Aut}_AB)$. Then we can consider $a^y$ as element of $A$    $$x^{a^y}=x^{y^{-1}ay}=x^{y^{-1}\varphi(a)y}=x^{\varphi(a)^y}, \,\, x\in B.$$
Now it is clear that
$$[b, a]^y=y^{-1}b^{-1}yy^{-1}b^ay=(b^{-1})^yb^{ay}=(b^y)^{-1}(b^{y})^{y^{-1}ay}=[b^y, a^y].$$

\section{$R$-nilpotent Groups}

Let $R$ be a group of operators for a group $G$ and
$$ K_0(G, R)=G\;\mathrm{and}\; K_n(G, R)=[K_{n-1}(G, R), R].$$
$$L_0(G, R)=1\;\mathrm{and}\; L_n(G, R)=\{ x\in G\,\,| \,\,[x, \alpha_1,\dots, \alpha_n]=1 \,\,\forall \alpha_1,\dots, \alpha_n\in R\}.$$
\begin{lem}\label{l1}Let $R$ be a group of operators for a group $G$ and $\mathrm{Inn}G\leq N_{\mathrm{Aut}G}(\mathrm{Aut}_RG)$. Then

$(1)$ $L_n(G, R)$ is a normal subgroup of $G$  for all natural $n$;

$(2)$ $L_n(G, R)=G$ for some natural $n$ if and only if $K_n(G, R)=1$.\end{lem}
\begin{proof} Let prove $(1)$. Let $x\in L_n(G, R)$. Note that for every $\alpha\in R$ and $y\in G$ there exists $\beta\in R$ with $\alpha=\beta^y$.  From $1=1^y=[x, \alpha_1,\dots, \alpha_n]^y=[x^y, \alpha_1^y,\dots, \alpha_n^y] \,\,\forall \alpha_1,\dots, \alpha_n\in R$ it follows that $[x^y, \alpha_1,\dots, \alpha_n]=1 \,\,\forall \alpha_1,\dots, \alpha_n\in R$. Hence $x^y
\in L_n(G, R)$. Thus $ L_n(G, R)$  is normal in $G$.

Let show that $ L_n(G, R)$ is a subgroup of $G$.  Since $G$ is finite, it is sufficient to show that if $x, y\in L_n(G, R)$, then $xy\in L_n(G, R)$. Let $\alpha\in R$. It is straightforward to check that $[xy, \alpha]=[x, \alpha]^y[y, \alpha]$. From $\mathrm{Inn}G\leq N_{\mathrm{Aut}G}(\mathrm{Aut}_RG)$ it follows that $[x, \alpha]^y[y, \alpha]=[x^y, \alpha^y][y, \alpha]$, where $\alpha^y\in R$. It means that $[xy, \alpha_1,\dots, \alpha_n]=[z, \beta_1,\dots, \beta_n][y, \alpha_1,\dots, \alpha_n]$, where $\beta_1,\dots,
\beta_n\in R$ and $z\in L_n(G, R)$ as a conjugate of $x$. Therefore $xy\in L_n(G, R)$. Thus  $L_n(G, R)\leq G$.

 Let prove $(2)$. Assume that $K_n(G, R)=1$. From $\{[x, \alpha_1,\dots, \alpha_n]\,|\,\forall x\in G$ and $\forall \alpha_1,\dots, \alpha_n\in R\}\subseteq K_n(G, R)$ it follows that $L_n(G, R)=G$.

Assume now that $L_n(G, R)=G$. Since $L_i(G, R)$ is a subgroup,  $[L_i(G, R), R]\leq L_{i-1}(G, R)$. From  $L_n(G, R)=G$ it follows that
 $K_i(G, R)\leq L_{n-i}(G, R)$. Thus $K_n(G, R)=1$.  \end{proof}
We shall call a group $G$ $R$-\emph{nilpotent} if $K_n(G, R)=1$  for some natural $n$. Hence if $R=G$, then $R$-nilpotent group $G$ is nilpotent and if $R=\mathrm{Aut}G$, then $R$-nilpotent group $G$ is autonilpotent by Lemma \ref{l1}.

\begin{thm}\label{th3.1}
Let $R$ be a group of operators for a group $G$. Then  $G$ is   $R$-nilpotent if and only if  $R$ stabilizes some chain of subgroups of $G$.\end{thm}

\begin{proof} Assume that $G$ is $R$-nilpotent. Then $1=K_n(G, R)$ for some $n$. From  $[K_i(G, R), R]=K_{i+1}(G, R)$ it follows that $xK_{i+1}(G, R)=x^\alpha K_{i+1}(G, R)$ for every $\alpha\in R$ and $x\in K_i(G, R)$. It means that $R$ stabilizes  $$1=K_n(G, R)\leq K_{n-1}(G, R)\leq\dots\leq K_0(G, R)=G.$$
Assume that  $R$ stabilizes the following chain of subgroups
$$1=G_n<G_{n-1}<\dots<G_0=G$$
Note tat $K_0(G, R)=G\leq G_0$. Assume that $K_i(G, R)\leq G_i$. Let show that $K_{i+1}(G, R)\leq G_{i+1}$. Since $x^\alpha G_{i+1}=x G_{i+1}$ for every $\alpha\in R$ and $x\in G_i$, we see that $$K_{i+1}(G, R)=[K_{i}(G, R), R]\leq [G_i, R]\leq G_{i+1}.$$
Hence $K_n(G, R)\leq G_n=1$. Therefore $G$ is $R$-nilpotent.
\end{proof}


 Theorem \ref{th1} follows from Theorem \ref{th3.1} when $R=\mathrm{Aut}G$. According  to \cite[A, 13.8(b)]{s8} every group $G$ is $\mathrm{F}(G)$-nilpotent, where $\mathrm{F}(G)$ is the Fitting subgroup of $G$.




\begin{thm}\label{t1} Let $p$ be a prime and $R$ be a group of operators for  a $p$-group $G$. Then $G$  is $R$-nilpotent if and only if $\mathrm{Aut}_RG$ is a $p$-group.\end{thm}

\begin{proof} $(a)$  \emph{If a $p$-group $G$ is $R$-nilpotent, then $\mathrm{Aut}_RG$ is a $p$-group.}

Note that $R$ stabilizes the chain of subgroups \[1=K_n(G, R)<K_{n-1}(G, R)<\dots<K_1(G, R)<K_0(G, R)=G.\] By \cite[Corollary 12.4A(a)]{s8},  $\mathrm{Aut}_RG= R/C_R(G)$ is a $p$-group.

$(b)$ \emph{Let $P$ be a $p$-group of operators of a $p$-group  $G$. Then $G$ has a $P$-admissible  maximal subgroup.}

 Note that maximal subgroups of $G$ are in one-to-one correspondence with maximal subgroups of $G/\Phi(G)$. It is well known that $G/\Phi(G)$ is isomorphic to the direct product of $n$ copies  of $Z_p$ for some natural $n$. So the number of maximal subgroups of $G/\Phi(G)$ is equal to the number of maximal subspaces of a vector space of dimension $n$ over $\mathbb{F}_p$. It is known (for example see \cite{amm}) that this number $k=(p^n-1)/(p-1)$. So the number  of maximal subgroups of $G$ is coprime to $p$. Note that if $M$ is a maximal subgroup of $G$, then
    \[M^x=\{m^x| m\in M\}\] is also maximal subgroup of $G$. Hence a $p$-group $P$ acts  on the set of maximal subgroups of $G$.    From $(k, p)=1$ it follows that $P$  has a fixed point  on it. It means that there exists a maximal subgroup $M$ of $G$ with $M^x=M$ for all $x\in P$.

 $(c)$ \emph{Let $P$ be a group of operators of a $p$-group  $G$. If $\mathrm{\mathrm{Aut}}_PG$ is a $p$-group, then $G$ has   $P$-composition series with simple factors.}

Note that the actions of $P$ and $\mathrm{Aut}_PG$ on $G$ are the same.   From $(b)$  it follows that every $P$-admissible  subgroup $M$ of $G$ has maximal $P$-admissible subgroup $N$. Hence  $G$ has   $P$-composition series with simple factors.

$(d)$  \emph{Let $P$ be a group of operators of a $p$-group  $G$. If $\mathrm{Aut}_RG$   is a $p$-group, then $G$ is $R$-nilpotent. }

 By $(c)$ $G$ has   $R$-composition series with simple factors. Let \[1=G_n<G_{n-1}<\dots<G_1<G_0=G\] be this series. Note that $G=K_0(G, R)\leq G$. Assume that we show that $K_i(G, R)\leq G_i$ for some $i$. Let show that  $K_{i+1}(G, R)\leq G_{i+1}$. Note that the order of $\mathrm{Aut}G_i/G_{i+1}\simeq \mathrm{Aut}Z_p\simeq Z_{p-1}$ is coprime to $p$. Hence $R$ acts trivially on  $G_i/G_{i+1}$, i.e. \[(gG_{i+1})^{-1}(gG_{i+1})^\alpha=[g, \alpha]G_{i+1}=G_{i+1}\]
for all $g\in G_i$ and $\alpha\in R$.
So  $[g, \alpha]\in G_{i+1}$   for all $g\in G_i$ and $\alpha\in R$. It means that
\[K_{i+1}(G, R)=[K_i(G, R), R]\leq [G_i, R]\leq G_{i+1}.\]
Thus $K_n(G, R)\leq G_n=1$. Hence $G$ is $R$-nilpotent.
\end{proof}

It is clear that $L_n(G, R)\subseteq L_{n+1}(G, R)$. Since $G$ is finite, we see that there is a natural $n$ such that $L_n(G)=L_{n+i}(G)$ for all $i>1$.  In this case let $L_n(G, R)=L_\infty(G, R)$. Note that if $R=\mathrm{Inn}G$, then $L_\infty(G, R)=\mathrm{Z}_\infty(G)$ the hypercenter of $G$, and if $R=\mathrm{Aut}G$, then $L_\infty(G, R)=L_\infty(G)$ the absolute hypercenter of $G$ .


\begin{thm}\label{t2.1}Let   $R$ be a group of operators for  a group $G$ with $\mathrm{Inn}G\leq\mathrm{Aut}_RG$ and   $g$ be a $p$-element of   $G$. Then $g\in L_\infty(G, R)$ if and only if $g^\alpha=g$ for every $p'$-element $\alpha$ of $R$.      \end{thm}


\begin{proof} From $\mathrm{Inn}G\leq\mathrm{Aut}_RG$ it follows that  $L_\infty(G, R)$ is nilpotent. Let $g$ be a $p$-element of $L_\infty(G, R)$. Note that if $g\in L_1(G, R)$, then  $[g, \alpha]=1$ or $g=g^\alpha$ for all    $p'$-elements $\alpha$ of $R$. Assume that if $g\in L_k(G, R)$, then  $g=g^\alpha$ for all    $p'$-elements $\alpha$ of $R$. Let show that if $g\in L_{k+1}(G, R)$, then  $g=g^\alpha$ for all    $p'$-elements $\alpha$ of $R$.

Since $L_\infty(G, R)$ is nilpotent, $[g, \alpha]$ is a $p$-element for all  $\alpha\in R$. Assume now that $\alpha\in R$ is a $p'$-element. From $g\in L_{k+1}(G, R)$ it follows that $[g, \alpha]\in L_k(G, R)$. By induction $[g, \alpha]^\alpha=[g, \alpha]$. Let $m$ be the order of $\alpha$.   From $g^\alpha=g[g, \alpha]$ it follows that $g=g^{\alpha^m}=g[g, \alpha]^m$ or $[g, \alpha]^m=1$. Since $[g, \alpha]$ is a $p$-element and $(p, m)=1$, we see that $[g, \alpha]=1$ or $g=g^\alpha$. Thus if $g\in L_{k+1}(G, R)$, then  $g=g^\alpha$ for all    $p'$-elements $\alpha$ of $R$.

From $L_\infty(G, R)=L_n(G, R)$ for some natural $n$ it follows that if $g$ is a $p$-element of  $L_\infty(G, R)$, then  $g^\alpha=g$ for every $p'$-element $\alpha$ of $R$.

Let $G_p$ be the set of all elements of $G$ such that   $g^\alpha=g$ for every $p'$-element $\alpha$ of $R$ and every $g\in G_p$. Note that if $x, y\in G_p$, then $xy\in G_p$. Hence $G_p$ is a subgroup of $G$. Let $g\in G_p$, $\alpha, \beta\in R$ and $\beta$ is be a $p'$-element. Then $\beta^{\alpha^{-1}}$ is a $p'$-element too. Hence
\[(g^\alpha)^\beta=g^{\alpha\beta\alpha^{-1}\alpha}=g^{\beta^{\alpha^{-1}}\alpha}=g^\alpha.\]
It means that $g^\alpha\in G_p$. Thus $G_p$ is a $R$-admissible  subgroup of $G$. Let $x$ be a $p'$-element of $G_p$. From $\mathrm{Inn}G\leq\mathrm{Aut}_RG$ it follows that
\[\alpha_x: g\rightarrow g^x\]
is a $p'$-element of $R$. Hence it acts trivially on $G_p$. It means that $x\leq \mathrm{Z}(G_p)$. Let $P$ be a Sylow $p$-subgroup of $G_p$. Then $P\trianglelefteq G_p$. It means that all $p$-elements of $G_p$ form a subgroup. So $P char G_p$. Hence $P$ is $R$-admissible. Now $\mathrm{Aut}_RP$ is a $p$-group.     It means that $P$ is $R$-nilpotent by Theorem \ref{t1}. So $K_n(P, R)=1$ for some natural $n$. Thus $[g, \alpha_1,\dots, \alpha_n]=1$ for any $g\in P$ and $\alpha_1,\dots, \alpha_n\in R$. Therefore $P\leq L_\infty(G, R)$.\end{proof}

\begin{cor}[Baer \cite{h1}] Let $p$ be a prime and $G$ be a group. Then a $p$-element $g$ of $G$ belongs to $\mathrm{Z}_\infty(G)$ if and only if it permutes with all $p'$-elements of $G$.        \end{cor}

\begin{thm}\label{t3.4}Let   $R$ be a group of operators for  a group $G$ with $\mathrm{Inn}G\leq\mathrm{Aut}_RG$. Then $G$ is $R$-nilpotent if and only if $\mathrm{Aut}_RP$ is a $p$-group for all $p$-subgroups $P$ of $G$ and   $p\in\pi(G)$.      \end{thm}




 \begin{proof} Assume that a group $G$ is $R$-nilpotent. Then $G=L_\infty(G, R)$ by Lemma \ref{l1}. Let  $P$ be a $p$-subgroup of $G$. Then $g=g^\alpha$ for all    $p'$-elements $\alpha$ of $R$ and all $g\in P$ by Theorem \ref{t2.1}. It means that  $\mathrm{Aut}_RP=N_{R}(P)/C_{R}(P)$ is a $p$-group.

   Assume that $\mathrm{Aut}_RP$ is a $p$-group for every  $p$-subgroup $P$ of $G$ and every $p\in\pi(G)$. Suppose that $G$ is non-nilpotent. So there is a Schmidt subgroup $S$ of $G$. Then $S$ has a normal $q$-subgroup $Q$ for some prime $q$ and there is a $q'$-element $x$ of $S$ with $x\not\in C_S(Q)$.  Since
\[\alpha_x: g\rightarrow g^x\]
is    a non-identity  inner automorphism of $G$ of  $q'$-order and $\mathrm{Inn}G\leq\mathrm{Aut}_RG$, $\mathrm{Aut}_RQ$ is not a $q$-group for a $q$-subgroup $Q$, a contradiction.

   Hence $G$ is nilpotent.     Let $P$ be a Sylow $p$-subgroup of $G$. Then $P$ contains all $p$-elements of $G$ and $P char G$. Since    $\mathrm{Aut}_RP=R/C_R(P)$ is a $p$-group,  $g=g^\alpha$ for all    $p'$-elements $\alpha$ of $R$ and all $g\in P$. From Theorem \ref{t2.1} it follows that $P\leq L_\infty(G, R)$. Hence $G\leq L_\infty(G, R)$.  Therefore $G= L_\infty(G, R)$. Thus $G$ is $R$-nilpotent by Lemma \ref{l1}.  \end{proof}

\begin{thm}\label{t3.5}
Let   $R$ be a group of operators for  a group $G$ with $\mathrm{Inn}G\leq\mathrm{Aut}_RG$. Then $G$ is   $R$-nilpotent if and only if it is  the direct product of its Sylow subgroups and $\mathrm{Aut}_RP$ is a $p$-group for every Sylow $p$-subgroup $P$ of $G$ and all   $p\in\pi(G)$.\end{thm}

\begin{proof}Assume that $G$ is $R$-nilpotent. From  $\mathrm{Inn}G\leq\mathrm{Aut}_RG$ it follows that $G$ is nilpotent. Hence it is the direct product of its Sylow subgroups. Note that $\mathrm{Aut}_RP$ is a $p$-group for every Sylow $p$-subgroup $P$ of $G$ and all $p\in\pi(G)$ by Theorem \ref{t3.4}.

The proof of the converse statement is the same as in the end of proof of Theorem \ref{t3.4}.
\end{proof}

\begin{proof}[Proof of Theorem \ref{th1}]
The automorphism group of a direct product of groups was described in \cite{adp}. In particular, if $G=P\times H$, where $P$ is a Sylow subgroup of $G$, then     $\mathrm{Aut}G=\mathrm{Aut}P\times\mathrm{Aut}H$ and   $\mathrm{Aut}_{\mathrm{Aut}G}P=\mathrm{Aut}P$. Now Theorem \ref{th1} directly follows from Theorem \ref{t3.5}.
\end{proof}

  \section*{Final Remarks}

  In \cite{i2, i3} it was shown that if $G$ has has $A$-composition series with prime indexes then $A$ is supersoluble.   Shemetkov \cite{i3} and Schmid \cite{i4} studied $\mathfrak{F}$-stable groups of automorphisms for  a (solubly) saturated formation $\mathfrak{F}$.


Let $\mathfrak{F}$ be a class of groups, $R$ be a group of automorphisms of a group $G$ and
 $H/K$ be a $R$-composition factor  of $G$. We shall call $H/K$ $R$-$\mathfrak{F}$-central if
$$H/K\leftthreetimes \mathrm{Aut}_RH/K\in\mathfrak{F}.$$
Hence if $R=\mathrm{Inn}G$, then   $R$-$\mathfrak{F}$-central factor is just $\mathfrak{F}$-central.

\begin{definition}
  We shall call a group $G$ auto-$\mathfrak{F}$-group if every $\mathrm{Aut}G$-composition factor of $G$ is $\mathrm{Aut}G$-$\mathfrak{F}$-central.
\end{definition}

\begin{pr}
  Describe the class of all autosupersoluble groups.
\end{pr}

\begin{pr}
  Describe the class of all auto-$\mathfrak{F}$-groups, where $\mathfrak{F}$ is a hereditary saturated formation.
\end{pr}




\subsection*{Acknowledgments}

I am grateful to A.\,F. Vasil'ev for helpful discussions.

\end{document}